\documentclass[11pt]{amsart}

\usepackage{amssymb}
\usepackage{a4wide}
\usepackage{tikz}
\usepackage[T1]{fontenc} 
\usepackage{dsfont}
\usepackage{mathrsfs} %for mathscr
\usepackage{mathtools} %for prescript
\usepackage{url}
\DeclareMathAlphabet{\mymathbb}{U}{bbold}{m}{n}

\newtheorem{thm}{Theorem}
\theoremstyle{definition}

\newtheorem{lem}{Lemma}
\newtheorem*{defn}{Definition}

\newtheorem{cor}{Corollary}

\newtheorem{fact}{Fact}

\newtheorem*{algorithm}{Algorithm}
\newtheorem*{rem}{Remark}
\newtheorem*{ques*}{Question}
\newtheorem*{prob*}{Problem}

\newcommand{\cC}{\mathcal{C}}

\newcommand{\sP}{\mathcal{P}}
\newcommand{\sN}{\mathcal{N}}

\newcommand{\bs}{\boldsymbol{\Sigma}}
\newcommand{\bp}{\boldsymbol{\Pi}}
\newcommand{\bP}{\boldsymbol{\Pi}}
\newcommand{\bS}{\boldsymbol{\Sigma}}

\newcommand{\res}{\restriction}
 \newcommand{\ww}{\omega^\omega}
 \newcommand{\sm}{\setminus}

\newcommand{\R}{\mathbb{R}}

\newcommand{\N}{\mathbb{N}}

\begin{document}

\title[The descriptive complexity of the set of Poisson generic numbers]{The descriptive complexity\\of the set of Poisson generic numbers}
\author[V. Becher \and S. Jackson \and D. Kwietniak \and W. Mance]{Ver\'onica Becher  
\and 
Stephen Jackson
\and 
Dominik Kwietniak \and Bill Mance}

\maketitle

%{\small \tableofcontents}

\begin{abstract}
  Let $b\ge 2$ be an integer. We show that the set of real numbers that are Poisson generic in base $b$ is $\bP^0_3$-complete in the Borel hierarchy of
  subsets of the real line.  Furthermore,  the set of real numbers that are Borel normal in base $b$   and not Poisson generic in base $b$ is complete for the class given by the differences between $\bP^0_3$ sets.  We also show that the effective  versions of these results hold in the effective Borel hierarchy.
  \end{abstract}

\noindent
{\bf Keywords}: Poisson generic numbers;  normal numbers; descriptive set theory;
\medskip
\medskip

\noindent
{\bf MSC Classification}: 03E15; 11U99; 11K16.
\medskip

% 	03E15  	Descriptive set theory
% 	11U99  Connectuoon between logic and number thery - 	None of the above, but in this section
% 11K16  	Normal numbers, radix expansions, Pisot numbers, Salem numbers, good lattice points, etc. [See also 11A63]

%\note{Along all the paper, we need to write real NUMBERS instead of sequences of symbols.  --V.}

\section{Introduction and statement of results}

Years ago Zeev Rudnick introduced Poisson generic real numbers:
a real number $x$ is Poisson generic in an integer base~$b\ge 2$,
if the counts of number of occurrences of words of length $k$ over the alphabet $\{0,1,\ldots,b-1\}$ appearing in the initial segments of the base $b$ expansion of $x$ tends to the Poisson distribution  with parameter $\lambda$ as $k\to\infty$ for every $\lambda>0$. That is, we look at the fraction of $k$ words appearing a given number of times among the first digits tends in distribution to the Poisson distribution with parameter $\lambda$ as $k\to \infty$. Peres and Weiss~\cite{weiss2020} proved that Lebesgue almost all real numbers are Poisson generic.  Their proof is presented in~\cite[Theorem 1]{ABM}. Poisson genericity implies (Borel) normality.

For the rest of the paper, given an integer $b\ge 2$,  
we identify real numbers in the unit interval $[0,1)$ with their base $b$ expansions, 
that is, we identify each $x\in [0,1)$ with a sequence $x_1x_2x_3\ldots$ with values in $\{0,1,\ldots,b-1\}$ such that
\[
x=\sum_{j=1}^\infty\frac{x_j}{b^j}
\]
and $x_j\neq 0$ for infinitely many $j\ge 1$. All real numbers in $[0,1)$ have at least one, and for all, but countably many real numbers the base $b$ expansion is unique.

In the sequel we  consider an integer $b\geq 2$ that we take as the given base. For a real number $x$ and an interval $A=[q,r]$ of real numbers (respectively $A=[q,r)$), where $1\le q<r$ we write $x\res A$ to denote the segment of the base-$b$ expansion of $x$ corresponding to positive integers in the interval $A$.   Many times instead of  writing $x\res[1,r]$ for some $r>1$ we write 
we write  $x\res r$ 
to denote  the initial segment of the base $b$ expansion of $x$  up to position $\lfloor r\rfloor$. 

Since Poisson genericity in base $b$ is a property that depends only of the tail of the base~$b$ representation of that real  number, the integer part of the number is irrelevant. Thus, we present our results  just for the real numbers in the unit interval, but they also hold when the unit interval is replaced by the real line.

\begin{defn}[Poisson generic number]
Let $\lambda$ be positive  real number. A real number  $x\in[0,1)$  is $\lambda$-Poisson generic in base $b$ if for every non-negative integer $j$ 
we have 
\[
\lim_{k\to \infty} Z^\lambda_{j,k} (x)= e^{-\lambda} \frac{\lambda^j}{j!},
\]
where \[ Z^\lambda_{j,k} (x) = \frac{1}{b^k} | \{  w \in \{0, \ldots (b-1)\}^k\colon w \text{ occurs } j
\text{ times in } x\res \lambda b^k+k\}|. \]
A real number $x$ is Poisson generic in base $b$ if it is $\lambda$-Poisson generic in base $b$
 for every positive real~$\lambda$.
\end{defn}

Let ${\mathcal P}_b$ be the set of real numbers that are Poisson generic  in base $b$.
It is easy to see that ${\mathcal P}_b$ is a Borel set. Our goal is to 
%show the following result concerning 
give the descriptive complexity of 
${\mathcal P}_b$. In other words, we would like to locate the exact position of ${\mathcal P}_b$ in the Borel hierarchy (both, lightface and boldface).

Recall that the Borel hierarchy for subsets of the real numbers is the stratification of the $\sigma$-algebra generated by the open sets with the usual topology.
 For references see Kechris's textbook~\cite{Kechris}. 

% The Borel hierarchy for subsets of 
% $ b^\omega$  is the stratification of the
% $\sigma$-algebra generated by the open sets with the usual  topology given by  the family of basic open sets $O_s= \{x\in b^\omega: x\res |s|= s\}$, for every finite sequence $s$ of digits in $\Z/b\Z$.

A set $A$ is $\bS^0_1$ if and only if $A$  is open and $A$ is 
$\bP^0_1$ if and only if $A$ is closed.  
$A$ is $\bS^0_{n+1}$ if and only if it is a
countable union of $\bP^0_{n}$ sets, and  $A$  is $\bP^0_{n+1}$ if and only if it is a countable
intersection of $\bS^0_n$ sets. 
 
A set $A$ is hard for a Borel class if  and only if every set in the class is reducible to $A$ by a continuous map.
A set $A$ is complete in a class if it is hard for this class and belongs to the class.
By Wadge's celebrated theorem, in spaces  like 
the real numbers with the usual interval topology, 
a $\bS^0_n$ set is $\bS^0_n$-complete if
  and only if it is not $\bP^0_n$.

When we restrict to
intervals with rational endpoints and 
computable countable unions and intersections, we obtain 
the  effective or lightface Borel hierarchy. 
One way to present the finite levels of the effective  Borel hierarchy
is by means of the arithmetical hierarchy of formulas in the language of second-order arithmetic. 
Atomic formulas in this language assert  algebraic identities between integers 
or membership of real numbers in intervals with rational endpoints.
A formula in the arithmetic hierarchy  involves only quantification over integers.
A formula  is $\Pi^0_0$ and $\Sigma^0_0$ if  all  its quantifiers are bounded.  
It is $\Sigma^0_{n+1}$ if it has the form $\exists x\, \theta$  where $\theta$ is $\Pi^0_n$,
and it  is $\Pi^0_{n+1}$ if it has the form $\forall x\, \theta$ where  $\theta$ is $\Sigma^0_n$.  

A set $A$ of real numbers 
is $\Sigma^0_n$ (respectively $\Pi^0_n$) in the effective Borel hierarchy if
and only if membership in that set is definable by a formula
which is $\Sigma^0_n$ (respectively
$\Pi^0_n$).
Notice that every $\Sigma^0_n$ set is $\bS^0_n$ and every  $\Pi^0_n$ set is $\bP^0_n$.
In fact, for every set $A$ in $\bS^0_n$ there is a  $\Sigma^0_n$ formula and real parameter 
such that membership in $A$ is defined by that $\Sigma^0_n$ formula relative to that real parameter.

A set $A$ is hard for an  effective Borel class if and only if every set in the class is
reducible to $A$ by a computable map.  
As before, $A$ is complete in an effective  class if it is hard for this class and belongs to the class.
Since computable maps are continuous, 
proofs of hardness in the effective hierarchy often yield proofs of hardness in general by relativization. 

The difference hierarchy over a pointclass is  generated by taking differences of sets.  In the sequel we are just interested in the class $D_2\text{-}\bP^0_3$ which 
consists of all the sets that are difference between two sets in  $\bP^0_3$. The class $D_2\text{-}\Pi^0_3$ is the effective counterpart.

Although the definition of Poisson genericity in a given base $b$ asks for  $\lambda$-Poisson genericity in base $b$  for every positive real $\lambda$, it suffices to consider $\lambda$-Poisson genericity in base $b$  for every positive rational  $\lambda$. This is proved in Lemma \ref{lemma:rat}.  Then, by the form of its definition, the set ${\mathcal P}_b$  
is  a $\Pi^0_3$ property, hence ${\mathcal P}_b$ is a Borel set appearing as $\bP^0_3$ set in the Borel hierarchy.
%{\color{red} Actually why? Definition asks for quantifier over all $\lambda>0$...}
%\note{I added a sentence here, you are welcome to improve it .  --V.}
We shall prove completeness. We  first prove the boldface case, and then we add the needed subtleties to prove the lightface case.
We start with the following result.

\begin{thm}\label{bold}
${\mathcal P}_b$ is $\bP^0_3$-complete.
\end{thm}

\begin{defn}[Borel normal number]
Let an integer $b\geq 2$. A real number $x$
is Borel normal in base $b$ if for every  block $w$ of digits in $\{0, \ldots (b-1)\}$,
\[
\lim_{n\to\infty}\frac{\text{the number of occurrences of $w$ in $x\res n$}}{n}= b^{-|w|}.
\]
\end{defn}
The set ${\mathcal P}_b$  of  real numbers that are Borel normal  in  base $b$ 
is  $\bP^0_3$-complete \cite{BS14,KiLinton}.
Every real 
 Poisson generic in base $b$ is Borel normal in  base $b$, see~\cite{weiss2020} or~\cite[Theorem 2]{BSH}.
We study the descriptive complexity of the difference set. 
Let $\sN_b$ be the set of  real numbers that Borel normal in base~$b$.

\begin{thm}\label{d2bold}
$\sN_b\sm \sP_b$ is $D_2\text{-}\bP^0_3$-complete.
\end{thm}

%We show that  the same result holds in the effective case. 
% When we say a set $A$ is (lightface) $\Gamma$-complete (for example 
% $\Gamma=\Pi^0_3$ or $\Gamma=D_2$-$\Pi^0_3$), 
% we mean it is in $\Gamma$ and every $\Gamma(x)$ set is reducible to $A$ by a function which is computable relative to $x$. 
The next two results are the lightface improvements of Theorems~\ref{bold} and~\ref{d2bold}.

\begin{thm}\label{light}
${\mathcal P}_b$ is $\Pi^0_3$-complete.
\end{thm}
\begin{thm} \label{d2light}
$\sN_b\sm \sP_b$ is $D_2\text{-}\Pi^0_3$-complete.
\end{thm}

Similarly to previous consequences of differences sets of normal numbers for Cantor series expansions being $D_2\text{-}\bP^0_3$-complete \cite{AireyJacksonManceComplexityCantorSeries}, 
Theorem~\ref{d2bold} imposes limitations on the relationship between $\sN_b$ and $\sP_b$. 
An immediate consequence of Theorem~\ref{d2bold} is that the set $\sN_n \sm \sP_b$ is uncountable. 
Also, since $\sN_b\sm \sP_b$ is $D_2\text{-}\bP^0_3$-complete, there cannot be a $\bS^0_3$ set $A$ such that $A \cap \sN_b=\sP_b$ (as otherwise, we would have $\sN_b\sm \sP_b=\sN_b\sm A \in \bP_0^3$, a contradiction).
Thus, no $\bS^0_3$ condition can be added to normality to give Poisson genericity. Equivalently, any time a $\bS^0_3$ set contains $\sP_b$, it must contain elements of $\sN_b\sm \sP_b$. As an application, consider the following definition of weakly Poisson generic: 

\begin{defn}[Weakly-Poisson generic number]
Say $x\in [0,1)$ with base $b$ expansion $(x_j)$ is weakly Poisson generic in base $b$ if for every $\epsilon>0$, every rational $\lambda$, and non-negative integer~$j$, we have that for infinitely many  $k$ that 
 $|Z^\lambda_{j,k} (x)-e^{-\lambda} \frac{\lambda^j}{j!}|<\epsilon$.
 \end{defn}
 
 Note that being Poisson generic in base $b$ implies being weakly-Poisson generic. 
 However, being weakly-Poisson generic is a $\bP^0_2$ condition. So, from Theorem~\ref{d2bold} we get the following:
 
\begin{cor}
For every base $b$ there is a base-$b$ normal number which is weakly Poisson generic but not Poisson generic. 
\end{cor}
 
As another application, consider the following version of discrepency. Suppose $f$ is a function assigning to each word $w\in b^{<\omega}$ and each positive integer $n$ a positive real number $f(w,n)$. Given $x\in [0,1)$ with base $b$ expansion $(b_j)$, say the $(w,n)$-discrepancy is 
$D(x,w,n)=| \frac{n}{b^{|w|}} - W( x\res n,w)|$, where $W(u,w)$ is the number of occurrences of $w$ in $u$. 
We say a real number $x$ has base $b$ $f$-large discrepancy if for all $w$ and all $n$ we have that 
$D(x,w,n)>f(w,n)$. The set of $x$ with $f$-large discrepancy, for any fixed $f$, is easily a $\bP^0_1$ set. 
The set of numbers that are Borel normal to base $b$ are exactly those for which the discrepancy of their initial segments of their expansion  in base $b$ goes to zero.
We conjecture that the Poisson generic numbers in base $b$ can not have very low discrepancy of their initial segments (for instance,  the  infinite de Bruijn sequences exist in bases $b\geq 3$,  they  satisfy that $Z^1_{1,k}=1$ for every $k$, hence they do not correspond to  Poisson generic numbers, and they have low discrepancy.)
However, we have the following, which states that the Poisson generic reals cannot be characterized as the set of normal numbers satisfying a large discrepancy condition. 

\begin{cor}
For every function $f$, the set of base-$b$ Poisson generic reals is not equal to the set of normal numbers with $f$-large discrepancy.
\end{cor}

There are also many other naturally occurring sets of real numbers are defined by conditions which make them $\bS^0_3$. 
Examples include countable sets, co-countable sets, the class BA of {\em badly approximable} numbers (which is a $\bS^0_2$ set), the Liouville numbers (which is a $\bP^0_2$ set),
and the set of $x\in [0,1]$ where a particular continuous function $f\colon [0,1]\to \R$ is not differentiable. 
In all these cases, the theorem implies that either the set omits some Poisson generic number, or else contains a number which is normal but not Poisson generic. Of course, many of these statements are easy to see directly, but the point is that they all follow immediately from the general complexity result, Theorem~\ref{d2bold}.

The set of real numbers whose expansion in {\em every}  integer base  is Poisson generic is of course~$\Pi^0_3$, but we do not know yet how to prove  that this set is $\Pi^0_3$-complete.
The set of real numbers  whose expansion in one base is $\lambda'$-Poisson generic but not $\lambda'$-Poisson generic, for different positive real numbers $\lambda $ and $\lambda'$, is $D_2\text{-}\Pi^0_3$ but we do not know if it is complete.

The result in the present note  contribute to the corpus of work on the descriptive complexity  
of properties of real numbers
that started with the questions of Kechris on the descriptive complexity of the set 
of Borel normal numbers.
He conjectured that set of
absolutely normal numbers (normal to all integer bases) is $\bp^0_3(\mathbb{R})$-complete.
 Ki and Linton \cite{KiLinton} gave the first  result towards solving the conjecture by showing that the set of numbers that are normal to base $2$ is $\bP^0_3$-complete.
Then  V.\ Becher,
P.\ A.\ Heiber, and T.\ A.\ Slaman \cite{BecherHeiberSlamanAbsNormal} settled Kechris'  conjecture.
Furthermore, V.\ Becher and T.\ A.\ Slaman \cite{BecherSlamanNormal}
proved that the set of numbers normal in at least one base is
$\bs^0_4(\mathbb{R})$-complete. In another direction, D.\ Airey, S.\ Jackson,
D.\ Kwietniak, and B.\ Mance \cite{AJKM} and, more generally K. Deka,  S.\ Jackson,
D.\ Kwietniak, and B.\ Mance in \cite{AJKMDynamics}
showed that for any dynamical system with
a weak form of the specification
property, the set of generic points for the system is $\bP^0_3$-complete. This result
generalizes the Ki-Linton result to many numeration systems other than the standard base $b$ one. In general, the Cantor series expansions are not covered in this generality, so D. Airey, S. Jackson, and B. Mance \cite{AireyJacksonManceComplexityCantorSeries} determined the descriptive complexity of various sets of normal numbers in these numeration systems.

\section{Boldface}

We write $\mu$ for the Lebesgue measure on the real numbers.
%uniform measure in $b^\omega$.
From Peres and Weiss metric theorem~\cite{weiss2020,ABM} asserting that $\mu$-almost all 
real numbers in the unit interval 
are Poisson generic in each integer base $b$, we have the following.

 For $\mu$ almost all real numbers $x$ in the unit interval the following holds.  Fix an integer base $b\geq 2$ and any $\alpha \in (0,1)$. 
 Then for any non negative integer $j$, and any $\epsilon >0$, for all large enough $k$ we have that 
 \[
 \Big|Z^{(1-\alpha)}_{j,k}(x)-e^{-(1-\alpha)} \frac{(1-\alpha)^j}{j!}\Big|<\epsilon.
 \]

\begin{proof}[Proof of Theorem~\ref{bold}]
Let $\cC=\{ z \in (\omega\sm\{ 0,1\})^\omega\colon \lim_i z(i)=\infty\}$. So, $\cC$ is $\bP^0_3$-complete. 
We define a continuous map $f \colon \ww\to (0,1)$
which reduces $\cC$ to $\sP_b$, that is, $f(z)\in\sP_b$ if and only if $z\in \cC$. 
Fix $z \in (\omega \sm\{0,1\})^\omega$. 
At step $i$ we define $f(z)\res [b^{k_{i-1}}, b^{k_i})$, where $\{ k_i\}$ is a sufficiently fast-growing sequence of positive integers.
Let 
\begin{align*}
B_i&:=[b^{k_{i-1}}, b^{k_i}), 
\\
B'_i&:=\left[b^{k_{i-1}}, \big(1-\frac{1}{z(i)}\big)b^{k_i} \right).
\end{align*}
The set $B'_i$ is non-empty as we may assume 
$k_i>2k_{i-1}$. We set 
\[
f(z)\res B'_i=x\res B'_i \text{ and } 
f(z)\res B_i\sm B'_i=0. 
\]

First suppose $z \notin \cC$, and fix $p \in \omega$ such that for infinitely many $i$ 
we have $z(i)=p$. Consider step $i$ in the construction of $f(z)$ for such an $i$. 
For any $\epsilon >0$, if $i$ is large enough then the number of words $w$ of length 
$k_i$ which occur in $x\res [1, (1-\frac{1}{z(i)})b^{k_i}]$
is at most 
\[
b^{k_i} (1-e^{-(1-\frac{1}{z(i)})}+\epsilon).
\]
So, the number $Z_i$ of words $w$ of length $k_i$ 
which occur in $f(z)\res b^{k_i}$ is at most 
\[
b^{k_i} \big(1-e^{-(1-\frac{1}{z(i)})}+\epsilon\big) + b^{k_{i-1}}.
\]
So, 
\[
\frac{1}{b^{k_i}} Z_i\leq  \big(1-e^{-(1-\frac{1}{p})}+2\epsilon\big)
\]
if $i$ is large enough 
using the fact that the $k_i$ grow sufficiently fast. On the other hand,
the Poisson estimate for the proportion of words of length $k_i$ occurring in a Poisson
generic sequence of length $b^{k_i}$ is $1-{1}/{e}$. Since $p$ is fixed, as $i$ 
gets large we have a contradiction. So, $f(z)$ is not $1$-Poisson generic.

Next suppose that $z \in \cC$. We show that $f(z)$ is Poisson generic in base $b$. 
Fix $\lambda>0$ and $\ell \in \omega$. Fix also $\epsilon >0$. Consider $k\in \omega$, 
and let $i$ be such that $k_{i-1}\leq  k  <k_i$. We show that for $k$ (and hence $i$) sufficiently 
large that 
$|Z^\lambda_{\ell,k}(f(z))- e^{-\lambda}\frac{\lambda^\ell}{\ell!}|<\epsilon$. 
Assume $i$ is large enough so that $\frac{1}{z(j)} <\epsilon$ for all $j \geq i-1$. 
First consider the case $\lambda \leq 1$. Note that, as $z(i)\geq 2$,
\[
b^k\leq \frac{1}{b} b^{k_i} \leq 
b^{k_i} \big(1-\frac{1}{z(i)}\big).
\]
We have that

\begin{equation} \label{eqn:e1}
\begin{split}
|\frac{1}{b^k} Z^\lambda_{\ell,k}(f(z))- \frac{1}{b^k} Z^\lambda_{\ell,k} (x)|
& \leq \frac{1}{b^k} \left( b^{k_{i-1}} \frac{1}{z(i-1)} + 6k + b^{k_{i-2}} \right)
\\ & 
\leq \frac{1}{z(i-1)} +\epsilon 
\\ &
\leq 2\epsilon.
\end{split}
\end{equation}
for $i$ large enough. We have used here the fact that  
\[
|Z^\lambda_{\ell,k}(f(z))- Z^\lambda_{\ell,k}(x)|
\]
is at most the number of words of length $k$ which appear in one of $f(z)\res b^k$, $x\res b^k$
at a position which overlaps the block $[b^{k_{i-1}}(1-\frac{1}{z(i-1)}),
b^{k_{i-1}})$, or else overlaps the block $[1,b^{k_{i-2}}]$, which gives the above estimate.

Consider now the case $\lambda >1$. If $\lambda b^k < b^{k_i} (1-\frac{1}{z(i)})$, then the same estimate above works. 
So, suppose $b^k \geq \frac{1}{\lambda} b^{k_i} (1-\frac{1}{z(i)})$. 
We may assume that 
\[
\lambda b^k< \frac{1}{2} b^{k_{i+1}} \leq (1-\frac{1}{z(i+1)}) b^{k_{i+1}}
\]
since $\lambda$ is fixed and the $k_i$ grow sufficiently fast (in particular $\frac{b^{k_{i+1}}}{b^{k_i}} 
\to \infty$). 
In this case we also count the number of words $w$ of length $k$ 
which might overlap the block of $0$s in $f(z)\res [b^{k_i}(1-\frac{1}{z(i)}), b^{k_i}]$. We then get 
\begin{equation*}
\begin{split}
\Big|\frac{1}{b^k} Z^\lambda_{\ell,k}(f(z))- \frac{1}{b^k} Z^\lambda_{\ell,k} (x)\Big|
& \leq \frac{1}{b^k} \Big( b^{k_{i-1}} \frac{1}{z(i-1)} + b^{k_i} \frac{1}{z(i)} +10k + b^{k_{i-2}} \Big)
\\ & 
\leq \frac{1}{z(i-1)} + \frac{b^{k_i}}{b^k} \frac{1}{z(i)} +\epsilon 
\\ &
\leq \frac{1}{z(i-1)}+ \lambda \frac{1}{1-\frac{1}{z(i)}} \frac{1}{z(i)}+\epsilon
\\ &
\leq 2\epsilon.
\end{split}
\end{equation*}
if $i$ is sufficiently large, since $\lambda$ is fixed and $z(i)\to \infty$.  
\end{proof}

For the proof of Theorem~\ref{d2bold} we require the following two lemmas.

\begin{lem} \label{fl}
Fix an integer  $b\geq 2$. Almost all real numbers in $(0,1)$
have the property that 
for any $\alpha$ of the form $\alpha=\frac{1}{2^\ell}$
we have 
\[
\lim_{i \to \infty} \frac{1}{b^{k_i}} |H_i|= (1-e^{-\alpha})(e^{-(1-\alpha)}),
\]
where $H_i$ is the set of words of length $k_i$ which occur in the base-$b$ expansion of $x$ with a starting position 
$[(1-\alpha)b^{k_i}, b^{k_i})$, but do not occur with a starting position in 
$[b^{k_{i-1}}, (1-\alpha)b^{k_i}]$.\\ 
In fact, this claim holds for any $x$ which is Poisson generic in base~$b$. 
\end{lem}

\begin{proof}
Let $x\in(0,1)$
be Poisson generic in base $b$ and fix $\alpha$ a negative power of $2$. 
Let 
\begin{itemize}
    \item $A_i$ be the set of  words of length $k_i$ occurring in $[b^{k_{i-1}}, b^{k_i})$.
    \item $C_i$ be the set of words of length $k_i$ occurring in  $[b^{k_{i-1}}, (1-\alpha)b^{k_i}))$. 
\end{itemize}

Clearly $C_i\subseteq A_i$.
The words which occur in $[(1-\alpha)b^{k_i}, b^{k_i})$ but not in $[b^{k_{i-1}}, (1-\alpha)b^{k_i}))$
are exactly the words which occur in $A_i$ but not $C_i$. 

Let
\begin{itemize}
    \item $A'_i$ be the set of  words that occur in $[1, b^{k_i})$ 
\item 
$C'_i$ be the set of  words that occur in $[1, (1-\alpha)b^{k_i})$.
\end{itemize}
Then 
\[
| | A_i \sm C_i|- |A'_i\sm C'_i| | \leq b^{k_{i-1}}.
\]
Since $x$ is Poisson generic in base $b$, for any $\epsilon >0$  we have that 
for all large enough~$i$ that 
\[
 \Big| \frac{1}{b^{k_i}} | A'_i| - (1-\frac{1}{e})\Big| <\epsilon.
\]
Similarly, as $x$ is Poisson generic in base $b$, and using $\lambda =1-\alpha$, we have that 
\[
\Big| \frac{1}{b^{k_i}} |C'_i| - (1- e^{-(1-\alpha)}) \Big| <\epsilon.
\]
So, 
\begin{equation*}
\begin{split}
\frac{1}{b^{k_i}} | A_i \sm C_i| & \leq \frac{1}{b^{k_i}} | A'_i \sm C'_i| + \frac{b^{k_{i-1}}}{b^{k_i}}
\\ & \leq (1-\frac{1}{e}) -(1- e^{-(1-\alpha)})+ \frac{b^{k_{i-1}}}{b^{k_i}}
+2\epsilon
\\ & \leq e^{-(1-\alpha)}(1-e^{-\alpha}) +3\epsilon. \qedhere
\end{split}
\end{equation*}
\end{proof}

%% \co{How do you obtain the  computation this Lemma?.
%% I see  this :
%% \[
%% \lim_{i \to \infty} \frac{1}{b^{k_i}} |H_i|  = \frac{1}{b^{k_i}} ( A_i - C_i )
%% \]
%% where\\ $A_i$ estimates 
%%  the  number of different words of  length $k_i$ that appear in $x[1, b^{k_i}]$ 
%% \[ 
%% A_i= b^{k_i} (1-e^{-1})
%% \]
%% and \\
%% $C_i$ estimates the    number of different words of  length $k_i$ 
%% that appear in $x[1, B_i^1 \cup B_i^2]$ 
%% \[
%% C_i= (1-\alpha) b^{k_i} (e^{1-\alpha} (1-\alpha) b^{k_i}
%% \]
%% }

Assume $x$ lies in the measure one set of Lemma~\ref{fl} and that the $k_i$ grow fast enough, then 
\[
\Big|\frac{1}{b^{k_i}}|H_i|- (1-e^{-\alpha})(e^{-(1-\alpha)})\Big|< \frac{1}{2^i}.
\]

A standard probability computation shows the following. 
\begin{lem} \label{pl}
There is a function $g \colon\omega \to \omega$ such that the following holds. 
Suppose $k_0<k_1<\cdots$ are such that $b^{k_i}-b^{k_{i-1}}>g(i-1)$ for all $i$. 
Then $\mu$-almost all $x \in (0,1)$ satisfy the following: for any $j \in \omega$, any
$w\in b^j$ and any $\epsilon >0$, 
for all large enough $i$, and any 
$n> g(i-1)$
\[
\Big| \frac{1}{n} W(x\res [b^{k_{i-1}}, b^{k_{i-1}}+n), w) -\frac{1}{b^j}\Big|<\epsilon
\]
where $W(s,w)$ is the number of occurrences of the word $w$ in $s$.
\end{lem}

\begin{proof}
We can take $g(n)=n$. Fix $j $ and $w \in b^j$, and fix $\epsilon>0$. It suffices to show that 
for almost all $x$ that for all large enough $i$ and any $n>g(i-1)=i-1$ that 
\[
\left| \frac{1}{n} W(x\res [b^{k_{i-1}}, b^{k_{i-1}}+n), w) -\frac{1}{b^j}\right|<\epsilon.
\]
There are constants $\alpha,\beta>0$ such that for all $n$, the probability that a string 
$s \in b^n$ violates the inequality 
$|\frac{1}{n} W(s,w) -\frac{1}{b^j}|<\epsilon$ is less than $\alpha e^{-\beta n}$.
So, the probability that an $x \in (0,1)$ violates 
$| \frac{1}{n} W(x\res [b^{k_{i-1}}, b^{k_{i-1}}+n), w) -\frac{1}{b^j}|<\epsilon$ for some 
$i\geq i_0$ and $n\geq i$ is at most 
\[
\sum_{i\geq i_0} \sum_{n \geq i} \alpha e^{- \beta n}
\leq \sum_{i \geq i_0} \alpha \frac{e^{-\beta i}}{1-e^{-\beta}} = \frac{\alpha e^{-\beta i_0}}{(1-e^{-\beta})^2}.
\]
Since this tends to $0$ with $i_0$, the result follows.
\end{proof}

We can now give the proof of the $D_2$-$\bP^0_3$ completeness of the difference set ${\mathcal N}_b\setminus {\mathcal P}_b$.

\begin{proof}[Proof of Theorem~\ref{d2bold}]
We fix a sufficiently fast growing sequence $k_0<k_1<\cdots$ as in Lemma~\ref{pl},
and then fix $x \in (0,1)$
to be Poisson generic in base $b$ (so that Lemma~\ref{fl} holds)
and also to be in the measure one set where Lemma~\ref{pl} holds for 
this sequence $(k_i)_{i\geq 0}$.

We let $C=\{ z \in \omega^\omega \colon z(2n)\to \infty\}$, and $D=\{ z \in \ww\colon z(2n+1)\to \infty\}$.
We define a continuous map  $f\colon \ww \to (0,1)$ which reduces $C\sm D$ to $\sN_b\sm \sP_b$.
The idea to define $f$ so that for $z \in \ww$, the even digits $z(2i)$ will control whether
$f(z)\in \sN_b$ and the odd digits $z(2i+1)$ will control whether $f(z)\in \sP_b$. 
When we wish to violate Poisson genericity, we will do so for $\lambda=1$ and $j=0$. 
We may assume without loss of generality that all $z(i)$ and all $k_i$ are positive powers of $2$. 

As before, at step $i$ we define $f(z)\res B_i$, where $B_i=[b^{k_{i-1}}, b^{k_i})$. 
Let 

\begin{equation*}
\begin{split}
B^1_i&:=[b^{k_{i-1}}, (1-\frac{1}{z(2i)}-\frac{1}{z(2i+1)}) b^{k_i})
\\ 
B^2_i&:=[(1-\frac{1}{z(2i)}-\frac{1}{z(2i+1)}) b^{k_i}, (1-\frac{1}{z(2i+1)})b^{k_i})
\\
B^3_i&:=[ (1-\frac{1}{z(2i+1)})b^{k_i}, b^{k_i})
\end{split}
\end{equation*}
So, 
\begin{align*}
|B^2_i|&= \frac{1}{z(2i)} b^{k_i}, 
\\
|B^3_i|&=\frac{1}{z(2i+1)} b^{k_i}. 
\end{align*}
We let 
\begin{align*}
f(z)\res B^1_i &:=x\res B^1_i,\\
%We set $
f(z)\res B^2_i&:=0,
\\
f(z)\res B^3_i&:= x\res [b^{k_{i-1}}, b^{k_{i-1}}+ |B^3_i|)=  x\res \Big[b^{k_{i-1}}, b^{k_{i-1}}+\frac{1}{z(2i+1)} b^{k_i}\Big). 
\end{align*}
We show that $f$ is a reduction from 
$C\sm D$ to $\sN_b\sm \sP_b$.

%% Let $H_i$ be the set of words of length $k_i$ which occur with a starting position in 
%% $B^3_i$ but do not occur with a starting position in $B^1_i\cup B^2_i$. 

First assume $z \notin C$, that is $z(2i)$ does not tend to $\infty$. Fix $\ell$
such that $z(2i)=\ell$ for infinitely many $i$. We easily have that $f(z)\notin \sN_b$. 
For example, if the digit $0$ occurs with approximately the right frequency $\frac{1}{b}$
in 
\[
f(z)\res [1, b^{k_{i-1}}+|B^1_i|)=\Big[0, b^{k_i}\Big(1-\frac{1}{z(2i)}-\frac{1}{z(2i+1)}\Big) \Big),
\]
then $0$ will occur with too large a frequency in 
\[
f(z)\res [1,b^{k_{i-1}}+|B^1_i|+|B^2_i|)=
\Big[0, b^{k_{i-1}}+|B^1_i|+ \frac{1}{\ell} b^{k_i}\Big).
\]
This is because $f(z)\res B^i_2=0$
and $|B^i_2|=\frac{1}{\ell} b^{k_i}$ for such $i$.

So we may henceforth assume that $z \in C$,
so $\frac{1}{b^{k_i}} |B^2_i|= \frac{1}{z(2i)}\to 0$. 
We  observe that 
this implies that $f(z) \in \sN_b$. 
This follows from Lemma~\ref{pl} and that we may assume 
$\lim_i \frac{g(i-1)}{b^{k_{i-1}}}=0$.

Now assume  that $z \in D$, so $z \notin C\sm D$. We show $f(z) \in \sP_b$, and so 
$f(z)\notin \sN_b \sm \sP_b$. Since we are assuming $z \in C$ also, we have 
$\lim_{i \to \infty} z(i)=\infty$. So, $\lim_{i \to \infty} \frac{1}{b^{k_i}} (|B^2_i|+|B^3_i|)=0$.
It then follows exactly as in Equation~\ref{eqn:e1} in the proof of 
Theorem~\ref{bold} that $f(z)\in \sP_b$. 

Assume next that $z \notin D$ (but $z \in C$ still). We show that $f(z)\notin \sP_b$,
which shows $f(z) \in \sN_b \sm \sP_b$. Fix $m$ so that for infinitely many $i$
we have $z(2i+1)=m$,
and we may assume $m$ is of the form $m=2^\ell$.  Recall $\frac{1}{b^{k_i}} |B^3_i|= \frac{1}{z(2i+1)}
=\frac{1}{2^\ell}$ for such $i$. We restrict our attention to this set of $i$ in the following
argument. If $f(z)$ were Poisson generic, then from Lemma~\ref{fl} we would have that 
for large enough $i$ in our set that 
\[
\frac{1}{b^{k_i}} |H_i|\approx (1-e^{-\alpha})(e^{-(1-\alpha)}),
\]
where $H_i$ is the set of words of length $k_i$ which occur in $f(z)$
with a starting position in $[(1-\alpha)b^{k_i}, b^{k_i})$, but do not occur in $f(z)$ with a starting position in 
$[b^{k_{i-1}}, (1-\alpha)b^{k_i})$. However, by the construction of $f(z)$ we have that every 
word which occurs in $[(1-\alpha)b^{k_i}, b^{k_i})$ also occurs in $[b^{k_{i-1}}, (1-\alpha)b^{k_i})$,
and so $|H_i|=0$. 
\end{proof}

\section{Lightface refinements}

The existence of computable Poisson generic real number was proved in~\cite[Theorem 2]{ABM}.
We start showing how to compute an instance
of a Poisson generic real number  in base~$b$.

% \begin{defn}[Bad sets]
% Fix integer base $b\geq 2$. 
% Define for $k\in\N,j\in\N_0$, $\lambda\in\R^+$, $\varepsilon\in\Q$,
% \begin{align*}
% Bad(\lambda, k, j,\varepsilon) &:= 
% \left\{ 
% x \in (0,1) :  
% | Z_{j,k}^{\lambda}(x) - \frac{e^{-\lambda} \lambda^j}{j!} | 
% > \varepsilon 
% \right\}
% \end{align*}
% \end{defn}

\begin{defn}[Values $N_n$ and sets $E_n$]
For each $n\geq 1$ define
\begin{align*}
N_n:=&b^{2n}
\\
E_n:=  &(0,1)  \setminus  \bigcup_{N_n \leq k< N_{n+1}}   Bad_k
\end{align*}
where
\begin{align*}
Bad_k:=&\bigcup_{j\in J_k} \bigcup_{\lambda\in L_k }Bad(\lambda, k, j,1/k))
\\
J_k := &\{0 ,\ldots , b^k-1\}.
\\
L_k := &\{p/q: q \in \{1,\ldots, k\}, p/q < k\} 
\\
 Bad(\lambda, k, j,\varepsilon) := &
 \left\{ 
 x \in (0,1) :  
 | Z_{j,k}^{\lambda}(x) - \frac{e^{-\lambda} \lambda^j}{j!} |  > \varepsilon 
 \right\}
\end{align*}
\end{defn}
Observe that each set $Bad_k$ is a finite union of intervals with rational endpoints.
Also each set  $E_n$ is a finite union of intervals with rational endpoints.

\begin{fact}\label{fact1}
There is $n_0$ such that  for  every $n$ greater than $n_0$,
$
\mu(E_n)> 1-\frac{1}{N_n^2}.
$
\end{fact}
\begin{proof}
By \cite[Proof of Theorem 2]{ABM}\label{lemma:bad}
there is $k_0$ such that for every $k\geq k_0$, for every $j\geq 0$,
\begin{align*}
\mu(Bad(\lambda, k, j,1/k) )&<2  e^{-\frac{ b^k}{2\lambda  k^4}} 
\end{align*}
and
\[
\mu(Bad_k) = \mu
\Big(
\bigcup_{j\in J_{k}}\bigcup_{\lambda\in L_{k}} Bad(\lambda,k,j,1/k)\Big)< 2 b^{k} k^3 e^{-b^{k}/ (2k^5)}. 
\]
Recall  $N_n=1/b^{2n}$. 
Let $n_0$ be the least integer greater than or equal to $k_0$ such that for every~$n\geq n_0$,

\begin{align*}
\mu(Bad_{N_n})&<\frac{1}{2 N_n^2}
\end{align*}
and
\begin{align*}
\mu\Big(
\bigcup_{N_n \leq k< N_{n+1}}Bad_k
\Big)&< 2 \mu(Bad_{N_n}).
\end{align*}
Since 
$ E_n=(0,1)  \setminus  \bigcup_{N_n \leq k< N_{n+1}}   Bad_k$,
 we have 
 \[
 \mu(E_n)\geq  1- 2 \mu(Bad_{{N_n}}).\]
Hence we obtain the wanted inequality,
\[
\mu(E_n) >1 - \frac{1}{N_n^2}.
\]
\end{proof}

Fact \ref{fact1} ensures that the  set $\bigcap_{n\geq n_0} E_n$ has positive measure.
Let see that $\bigcap_{n\geq n_0} E_n$ consists  entirely of Poisson generic real numbers for base~$b$.
Suppose that $x$ is not Poisson generic for base $b$. By Lemma \ref{lemma:rat} $x$ is not $\lambda$-Poisson generic in base $b$ for some positive rational $\lambda$.
Then, there is  a positive $\varepsilon$ and a non-negative integer $j$ such that 
for infinitely many $n$s, 
\[
\Big| Z_{j,n}^{\lambda}(x) - \frac{e^{-\lambda} \lambda^j}{j!} \Big| > \varepsilon.
\]
Let  $n_1=n_1(\lambda,\varepsilon,j)$ be the smallest 
such that $\lambda\in L_{n_1}, j\in J_{n_1}, \varepsilon\geq 1/n_1$.
Since sets $J_n$ and $L_n$ are subset increasing in $n$, for every $n\geq n_1$ we have $\lambda\in L_n$ and $j\in J_n$.
And since $ \varepsilon> 1/n_1$ we have  $\varepsilon>1/n$, for every $n>n_1$.
Then, for  infinitely many values of $n$ greater than or equal to $n_1$, $x\in Bad_n$.
Hence, for infinitely many values of $n$,  $x\not\in E_n$, and thus $x\not\in \bigcap_{n\geq n_0} E_n$.

The  following algorithm  is an adaptation of Turing's algorithm for computing an absolutely normal number (see \cite{BFP}).  We modified it to obtain a real that is Poisson generic  in base~$b$.

\begin{algorithm}\label{algorithm}\em
Let $n_0$ be determined by Fact \ref{fact1}.
Let $I_{n_0}:=(0,1)$.
\\At each step  $n>n_0$,
divide $I_{n-1}$ in $b$ equal parts $I^0_{n-1}, I^1_{n-1}, \ldots , I^{b-1}_{n-1}$.
\\Let $v$ be the smallest in $\{0, .., (b-1)\}$ such that 
$\displaystyle\mu(I^v_{n-1} \cap E_{n}) > \frac{1}{N_n}  $.
\\$I_n:=I^v_{n-1}$.
\\
The  $n$-digit in the base-$b$ expansion of $x$ is  the digit  $v$.
\end{algorithm}

\begin{rem}Observe that the number $x$ computed by the algorithm ensures that for each $n\geq n_0$, $x\in I_n\cap E_n$. Since the intervals $I_n$ and rested, we have 
\[
x\in I_n\cap \Big(
\bigcap_{n_0\leq m\leq n}E_m\Big),
\]
 where  $E_n=(0,1)\setminus \bigcup_{N_{n}\leq k\leq N_{n+1}} Bad_k$ with  $N_n={b^{2n}}$.
Thus, to define $x\res n$ the algorithm looks at all the possible continuations up to $x\res b^{N_{n+1}}$.
\end{rem}

We  prove that the number $x$ produced by the algorithm  is indeed Poisson generic for base~$b$.
The algorithm defines a sequence of intervals $(I_n)_{n\geq n_0}$ such that $I_n=\left(\frac{a}{b^n}, \frac{a+1}{b^n}\right)$ for some $a\in\{0, \ldots, b^n-1
\}$, $I_{n+1}\subseteq I_n$ and $\mu(I_n)=b^{-n}$.
The number $x$ defined is the unique element in $\displaystyle\bigcap_{n\geq n_0}I_n$.
We first prove  that for every  $n\geq n_0$, 
\[
\mu\Big( I_n \cap \bigcap_{i=n_0}^n E_i\Big) >0.
\]
To show this we prove by induction on $n$, 
\[
\mu\Big(I_n \cap \bigcap_{i=n_0}^n E_n \Big)>\frac{1}{N_n}.
\]

{\em Base case.} For $n_0$ it  is immediate because   $I_{n_0}=(0,1)$, so $\mu(I_{n_0})=1$ and 
\[
\mu(E_{n_0})>1-\frac{1}{N_{n_0}^2}>\frac{1}{N_{n_0}}.
\]
{\em Inductive case}. Assume the inductive hypothesis
\[
\mu\Big( I_{n} \cap \bigcap_{i=n_0}^n E_i \Big) >\frac{1}{N_n}.  
\]
Let's see it  holds  for $n+1$. Using the inductive hypothesis and Fact \ref{fact1}, we have
\begin{align*}
\mu\Big(  I_{n} \cap \bigcap_{i=n_0}^{n+1} E_i \Big)
&=
\mu\Big( \Big( I_{n} \cap \bigcap_{i=n_0}^n E_i \Big)\cap E_{n+1}\Big)
\\
&>
\mu  \Big(I_n \cap \bigcap_{i=n_0}^n E_i \Big)  -  \mu((0,1)- E_{n+1})
\\
&>  \frac{1}{N_n}  - \frac{1}{{N _n}^2}
\\
%&> 1/N_n -  2 b^{N_{n+1}} N_{n+1}^3 e^{-b^{N_{n+1}} / 2(N_{n+1})^5}
%\\
&> \frac{b}{N_{n+1}}.
\end{align*}
Then, it is impossible that for each of the  $b$ possible $v$s, $v=0, v=1, \ldots, v=(b-1)$,
\[
\mu\Big(I^v_{n} \cap \bigcap_{i=1}^{n+1} E_i \Big) \leq \frac{1}{N_{n+1}}.
\]
So,  there is at least one $v\in\{0, \ldots, (b-1)\}$ such that 
\[
\mu\Big(I^v_{n} \cap \bigcap_{i=n_0}^{n+1} E_i \Big) >\frac{1}{N_{n+1}}.
\]
Since the algorithm  sets $I_{n+1}$  to be  the leftmost $ I^v_{n}$ with this property,
we have
\[
I_{n+1} \cap \bigcap_{i=n_0}^{n+1} E_i > \frac{1}{N_{n+1}}.
\]
We conclude that $x\in \bigcap_{n\geq n_0} E_n$.
%This implies that  $x\not\in \bigcup_{n\geq 0}Bad_n$.
%By Lemma~\ref{lemma:out} 
So $x$ is $\lambda$-Poisson generic in base $b$ for all positive 
rational~$\lambda$.
Finally, to conclude that 
 $x$ is $\lambda$-Poisson 
generic in base $b$ for every positive real $\lambda$, 
hence Poisson generic in base $b$ we need the following lemma.

\begin{lem}[adapted from~\cite{ABM}] \label{lemma:rat}
Let integer $b\geq 2$.
If $x\in (0,1)$ is $\lambda$-Poisson generic  in base $b$ for all positive rational
 $\lambda$ then $x$ is Poisson generic in base $b$.
\end{lem}

\begin{proof}
For each $x\in (0,1)$
and for each $k\in \mathbb N$, 
on  the space of words of  length $k$
with uniform measure %product measure $\mu^k$,
define the integer-valued random measure  
$M^x_k =M_k^{x}(  v)$ on the  real half-line $\R^+=[0,+\infty)$ by setting for all Borel sets $S\subseteq \R^+$,
\[
M_k^{x}(S)(  v) := \sum_{p \in \N \cap b^k S} I_p(x,  v),
\]
where $I_p$ is the indicator function that $v$ occurs in $x$ at position $p$ and 
$ \N\cap b^k S$  denotes the set of integer values in  
$\{ b^k s: s\in S\}$.  Then, $M^x_k(\cdot)$ is a point process on $\R^+$.
The function~$Z_{j,k}^{\lambda}(x)$can be formulated  in terms of 
of $M^{x}_k(S)$ for the sets  $S=(0,\lambda]$, as follows:
\begin{align*}
Z_{j,k}^{\lambda}(x)
&=\frac{1}{b^k}
\#\{   v\in \{0,\ldots, (b-1)\}^k: M^{x}_k((0,\lambda])(  v)=j)\}.
\end{align*}
Observe that for every pair of positive reals  $\lambda,\lambda'$, with~$\lambda < \lambda'$, 
\[ 
M_k^x((0,\lambda'])(  v) 
- M_k^x((0,\lambda])(  v) = 
\sum_{p \in \N \cap b^k [\lambda, \lambda')} 
I_p(x,  v ).
\] 

The classical total variation distance $d_{TV}$ between two probability measures $P$ and $Q$ on a 
$\sigma-$algebra~$\mathcal{F}$ is defined via
\[
d_{TV}(P,Q):=\sup_{ A\in \mathcal{F}}\left|P(A)-Q(A)\right|.
\]
For a random variable $X$ taking values in $\R$, 
the distribution of~$X$ 
is the probability measure~$\mu_X$ on~$\R$ defined as the push-forward 
of the probability measure on the sample space of~$X$.
 The total variation distance between two random variables~$X$ and $Y$ is simply 
$d_{TV}(X,Y) = d_{TV}(\mu_X, \mu_Y).$
Hence, the total distance variation 
\[
d_{TV}(M_k^x((0,\lambda']), M_k^x((0,\lambda]))
\leq \frac{1}{b^k}\#(\N \cap b^k [\lambda, \lambda')  )
= \lambda' - \lambda + {\rm O}(b^{-k}).
\] 
Also observe that  $d_{TV}(Po({\lambda'}), Po({\lambda})) \to 0$ 
as $\lambda \to \lambda'$.
From these two observations and the fact that the rational numbers are a dense subset of the real
numbers we conclude that 
being $\lambda$-Poisson
generic for every positive rational $\lambda$  implies Poisson generic.
\end{proof}

The proofs of Theorems~\ref{light} and~\ref{d2light} are very similar to those of Theorems~\ref{bold} and~\ref{d2bold}.
However we now include what is needed to prove the lightface results.
We now start we a computable Poisson generic number in base $b$ that we obtain with the Algorithm above, and we computably determine the sequence of values $(k_i)_{i\geq 1}$ using the input sequence $z\in\omega^\omega$.

\begin{proof}[Proof of Theorem~\ref{light}]

Let $\cC=\{ z \in \omega^\omega\colon \lim_{i\to\infty} z(i)=\infty\}$. 
So, $\cC$ is $\Pi^0_3$-complete. 
We define a computable map $f \colon \ww\to (0,1)$ which 
reduces $\cC$ to ${\mathcal P}_b$. Fix $z \in\ww$. 
At step $i$, let  $k_i$ be  the least  integer  
such that  
$k_i>k_{i-1}$, 
and $k_i > z(i)$.
Fix $k_0=0$. For $i>0$, we define $f(z)\res [b^{k_{i-1}}, b^{k_i})$ as follows. 
Let 
\begin{align*}
B_i&:=[b^{k_{i-1}}, b^{k_i})
\\ %and let 
B'_i& :=\Big[b^{k_{i-1}}, \Big(1-\frac{1}{z(i)}\Big)b^{k_i} \Big).
\end{align*}
We set 
\[
f(z)\res B'_i:=x\res B'_i,  \text{ and } 
f(z)\res B_i\sm B'_i:=0. 
\]

First suppose $z \notin \cC$, and fix $\ell \in \omega$ such that for infinitely many $i$ 
we have $z(i)=\ell$. Consider step $i$ in the construction of $f(z)$ for such an $i$. 
For any $\epsilon >0$, if $i$ is large enough then the number of words $w$ of length 
$k_i$ which occur in $x\res [1, (1-\frac{1}{z(i)})b^{k_i}]$
is at most 
\[
b^{k_i} (1-e^{-(1-\frac{1}{z(i)})}+\epsilon).
\]
Then, the number $Z_i$ of words $w$ of length $k_i$ 
which occur in $f(z)\res b^{k_i}$ is at most 
\[
b^{k_i} (1-e^{-(1-\frac{1}{z(i)})}+\epsilon). 
\]
So, 
\[
\frac{1}{b^{k_i}} Z_i\leq  (1-e^{-(1-\frac{1}{\ell})}+2\epsilon) .
\]
% if $i$ is large enough using the fact that the $k_i$ grow sufficiently fast. 
On the other hand,
the Poisson estimate for the proportion of words of length $k_i$ occurring in an initial segment of length $b^{k_i}$ is $1-{1}/{e}$. Since $\ell$ is fixed, as $i$ 
gets large we have a contradiction. So, $f(z)$ is not $1$-Poisson generic n base $b$.

Next suppose that $z \in \cC$. We show that $f(z)$ is Poisson generic in base $b$. 
Fix a positive  rational $\lambda$ and   $\epsilon >0$.
Consider any  $k\in \omega$, large enough so that the following  holds:

Let  $i$  be such that  $k_{i-1}\leq  k  <k_i$,
\begin{itemize}
\item  if     $\lambda= \frac{p}{q}$ then $k_{i-1}\geq   q$,

\item  $k_i > \frac{1}{\epsilon}$,

%\log_2 \left(
%\frac{1}{\epsilon}
%\right)$ 
%\frac{1}{k_{i-1}} < \varepsilon$,

\item  
for all $s \geq i-1$ we have $\frac{1}{z(s)} <\epsilon$.

%\item   $b^{k_i} (1-\frac{1}{z(i)}) > \lambda b^k $
\end{itemize}

We show that for  any such $k$, and for every  non negative $j$ less than $b^k$,
\linebreak
 $|Z^\lambda_{j,k}(f(z))- e^{-\lambda}\frac{\lambda^j}{j!}|<\epsilon$. 
First consider the case $\lambda \leq 1$. Fix $j$.
We have that 
\begin{equation*}
\begin{split}
\Big|\frac{1}{b^k} Z^\lambda_{j,k}(f(z))- \frac{1}{b^k} Z^\lambda_{j,k} (x)\Big|
& \leq \frac{1}{b^k} \Big( b^{k_{i-1}} \frac{1}{z(i-1)} + 2k + \sum_{m<i-1} b^{k_m} \Big)
\\ & 
\leq \frac{1}{z(i-1)} +\epsilon 
\\ &
\leq 2\epsilon
\end{split}
\end{equation*}
for $i$ large enough.
 We have used here the fact that  $|Z^\lambda_{j,k}(f(z))- Z^\lambda_{j,k}(x)|$
is at most the number of words of length $k$ which appear in one of $f(z)\res b^k$, $x\res b^k$
but not the other. Such a word must overlap the block of $0$s in $f(z)\res [b^{k_{i-1}}(1-\frac{1}{z(i-1)}),
b^{k_{i-1}})$, or else overlap $[1,b^{k_{i-2}}]$, which gives the above estimate.

Consider now the case $\lambda >1$. If $\lambda b^k < b^{k_i} (1-\frac{1}{z(i)})$, then the same estimate above works. 
So, suppose $b^k \geq \frac{1}{\lambda} b^{k_i} (1-\frac{1}{z(i)})$. 
In this case we also count the number of words $w$ of length $k$ 
which might overlap the block of $0$s in $f(z)\res [b^{k_i}(1-\frac{1}{z(i)}), b^{k_i}]$. We then get 
\begin{equation*}
\begin{split}
\Big|\frac{1}{b^k} Z^\lambda_{j,k}(f(z))- \frac{1}{b^k} Z^\lambda_{j,k} (x)\Big|
& \leq \frac{1}{b^k} \Big( b^{k_{i-1}} \frac{1}{z(i-1)} + b^{k_i} \frac{1}{z(i)} +3k + \sum_{s<i-1} b^{k_s} \Big)
\\ & 
\leq \frac{1}{z(i-1)} + \frac{b^{k_i}}{b^k} \frac{1}{z(i)} +\epsilon 
\\ &
\leq \frac{1}{z(i-1)}+ \lambda \frac{1}{1-\frac{1}{z(i)}} \frac{1}{z(i)}+\epsilon
\\ &
\leq 2\epsilon
\end{split}
\end{equation*}
if $i$ is sufficiently large, since $\lambda$ is fixed and $z(i)\to \infty$.  
\end{proof}

We can now prove the $D_2$-$\Pi^0_3$-completeness of the difference set ${\mathcal N}_b\setminus {\mathcal P}_b$.

\begin{proof}[Proof of Theorem~\ref{d2light}]
The proof  is exactly as that of Theorem~\ref{d2bold} except that  we start with a computable 
real  $x$  and we determine the  
sequence $(k_i)_{i\geq 1}$  using the input sequence $z\in\omega^\omega$.
Let $x$ be the number obtained by  the Algorithm.

Let 
${C}:=\{ z \in \omega^\omega \colon z(2n)\to \infty\}$
and 
${D}:=\{ z \in \ww\colon z(2n+1)\to \infty\}$.
We  assume without loss of generality that all $z(i)$ are powers of $2$. 
We define a computable map $f \colon \ww\to (0,1)$ which reduces 
${ C}\sm { D}$ to $\sN_b\sm \sP_b$.
Fix $z \in\ww$. 
At step $i$, let  $k_i$ be  the least  power of $2$
such that  
 $k_i>k_{i-1}$, 
 and $k_i > z(i)$.
We define $f$ so that for $z \in \ww$, the even digits $z(2i)$ will control whether
$f(z)\in \sN_b$ and the odd digits $z(2i+1)$  control whether $f(z)\in \sP_b$. 
When we wish to violate Poisson genericity, we  do so for $\lambda=1$ and $j=0$. 

As in the proof of Theorem~\ref{d2bold}, 
at step $i$ we define $f(z)\res B_i$, where $B_i=[b^{k_{i-1}}, b^{k_i})$. 
Let 
\begin{align*}
f(z)\res B^1_i&:=x\res B^1_i\\
f(z)\res B^2_i&:=0\\
f(z)\res B^3_i&:= x\res \Big[b^{k_{i-1}}, b^{k_{i-1}}+\frac{1}{z(2i+1)} b^{k_i}\Big)
\end{align*}
where
\begin{align*}
B^1_i&:=\Big[b^{k_{i-1}}, \big(1-\frac{1}{z(2i)}-\frac{1}{z(2i+1)}\big) b^{k_i}\Big)
\\ 
B^2_i&:=\Big[\big(1-\frac{1}{z(2i)}-\frac{1}{z(2i+1)}\big) b^{k_i}, \big(1-\frac{1}{z(2i+1)})b^{k_i}\big)\Big)
\\
B^3_i&:=\Big[ (1-\frac{1}{z(2i+1)})b^{k_i}), b^{k_i}\Big).
\end{align*}
Notice that  $|B^2_i|= \frac{1}{z(2i)} b^{k_i}$, and $|B^3_i|=\frac{1}{z(2i+1)} b^{k_i}$.

We show that $f$ is a reduction from 
$C\sm D$ to $\sN_b\sm \sP_b$.
First assume $z \notin C$, that is $z(2i)$ does not tend to infinity when $i$ goes to infinity. 
Fix $\ell$
such that $z(2i)=\ell$ for infinitely many~$i$. We easily have that $f(z)\notin \sN_b$. 
For example, if the digit $0$ occurs with approximately the right frequency $\frac{1}{b}$
in $f(z)\res [0, b^{k_{i-1}}+|B^1_i|)=[0, b^{k_i}(1-\frac{1}{z(2i)}-\frac{1}{z(2i+1)}) )$,
then $0$ will occur with too large a frequency in 
\[
f(z)\res \Big[1,b^{k_{i-1}}+|B^1_i|+|B^2_i|)=
[0, b^{k_{i-1}}+|B^1_i|+ \frac{1}{\ell} b^{k_i}\Big).
\]
We use here that $\frac{1}{b^{k_i}}\sum_{k<i} b_k \to 0$. 
 This is because $f(z)\res B^i_2=0$
and $|B^i_2|=\frac{1}{\ell} b^{k_i}$ for such $i$.

Now assume that $z \in C$,
so $\frac{1}{b^{k_i}} |B^2_i|= \frac{1}{z(2i)}\to 0$.
Then,  we have  $f(z) \in \sN_b$. 
This follows from 
% Lemma~\ref{pl} and the fact that we may assume that $\lim_{i\to\infty} \frac{g(i-1)}{b^{k_{i-1}}}=0$.
the the definition of $f$ and the fact that by that $x$ is Borel normal to base $b$, see \cite[Theorem 2]{BSH}.

Assume first that $z \in D$, so $z \notin C\sm D$. We show $f(z) \in \sP_b$, and so 
$f(z)\notin \sN_b \sm \sP_b$. Since we are assuming $z \in C$ also, we have 
$\lim_{i \to \infty} z(i)=\infty$. So, $\lim_{i \to \infty} \frac{1}{b^{k_i}} (|B^2_i|+|B^3_i|)=0$.
It then follows exactly as  in the proof of 
Theorem~\ref{light} that $f(z)\in \sP_b$. 

Assume next that $z \notin D$ (but $z \in C$ still). We show that $f(z)\notin \sP_b$,
which shows $f(z) \in \sN_b \sm \sP_b$. Fix $m$ so that for infinitely many $i$
we have $z(2i+1)=m$,
and  $m$ is of the form $m=2^\ell$.  Recall $\frac{1}{b^{k_i}} |B^3_i|= \frac{1}{z(2i+1)}
=\frac{1}{2^\ell}$ for such $i$. We restrict our attention to this set of $i$ in the following
argument: 
If $f(z)$ were Poisson generic, then from Lemma~\ref{fl} we would have that 
for large enough $i$ in our set that 
\[
\frac{1}{b^{k_i}} |H_i|= (1-e^{-\alpha})(e^{-(1-\alpha)}),
\]
where $H_i$ is the set of words of length $k_i$ which occur in $f(z)$
with a starting position in $[(1-\alpha)b^{k_i}, b^{k_i})$, but do not occur in $x$ with a starting position in 
$[b^{k_{i-1}}, (1-\alpha)b^{k_i})$. However, by the construction of $f(z)$ we have that every 
word which occurs in $[(1-\alpha)b^{k_i}, b^{k_i})$ also occurs in $[b^{k_{i-1}}, (1-\alpha)b^{k_i})$,
and so $|H_i|=0$. This completes the proof of Theorem~\ref{d2light}.
\end{proof}
\bigskip
\bigskip

\noindent
{\bf Acknowledgements.} V.~Becher is supported by grant PICT 2018-2315 of Agencia Nacional de Promoción Científica y Tecnológica de Argentina. S.~Jackson is supported by NSF grant DMS-1800323. W.~Mance is  supported by grant 2019/34/E/ST1/00082 for the project ``Set theoretic methods in dynamics and number theory,'' NCN (The National Science Centre of Poland). D.~Kwietniak is supported by NCN (the National Science Centre, Poland) Preludium Bis project no. 2019/35/O/ST1/02266.
\bigskip

\bibliographystyle{plain}

%\bibliography{main}
\bigskip
\bigskip

\footnotesize 

\noindent
Ver\'onica Becher \\
 Departamento de  Computaci\'on, Facultad de Ciencias Exactas y Naturales \& ICC  \\
 Universidad de Buenos Aires \&  CONICET  Argentina\\  
 {\tt  vbecher@dc.uba.ar}
\medskip
\medskip

\noindent
Stephen Jackson 
\\
Department of Mathematics\\ 
University of North Texas, General Academics Building 435\\
1155 Union Circle,  \#311430, Denton, TX 76203-5017, USA
\\
{\tt Stephen.Jackson@unt.edu}
\medskip
\medskip

\noindent
Dominik  Kwietniak\\
Faculty of Mathematics and Computer Science\\
Jagiellonian University in Krak\'ow, ul. {\L}ojasiewicza 6, 30-348 Krak\'ow, Poland\\
{\tt dominik.kwietniak@uj.edu.pl }
\medskip
\medskip

\noindent
Bill Mance 
\\
Faculty of Mathematics and Computer Science\\ 
Adam Mickiewicz University, ul. Umultowska 87, 61-614 Pozna\'{n}, Poland
\\{\tt william.mance@amu.edu.pl}

\end{document}